\documentclass[12pt]{amsart}

\usepackage{amssymb}

\usepackage{xcolor,graphics}

\newcommand{\mcA}{\mathcal{A}}

\newcommand{\ZZ}{\mathbb{Z}}
\newcommand{\RR}{\mathbb{R}}
\newcommand{\va}{\mathbf{a}}
\newcommand{\vx}{\mathbf{x}}
\newcommand{\valpha}{\boldsymbol{\alpha}}
\newcommand{\vbeta}{\boldsymbol{\beta}}
\newcommand{\vtheta}{\boldsymbol{\theta}}
\newcommand{\vtau}{\boldsymbol{\tau}}
\newcommand{\vv}{\mathbf{v}}
\newcommand{\tDelta}{\widetilde{\Delta}}
\newcommand{\tPsi}{\widetilde{\Psi}}
\newcommand{\tI}{\widetilde{I}}
\newcommand{\tK}{\widetilde{K}}

\newcommand{\mes}{\operatorname{mes}}

\renewcommand{\Re}{\operatorname{Re}}

\newcommand{\pder}[2]{\frac{\partial #1}{\partial #2}}

\newtheorem{theorem}{Theorem}
\newtheorem{lemma}{Lemma}
\newtheorem{cor}{Corollary}

\title[Distribution of polynomial discriminants]{On the distribution of polynomial discriminants: totally real case}
\author[D.~V.~Koleda]{Denis V. Koleda}
 \address{Denis V. Koleda,
 Institute of Mathematics,
 National Academy of Sciences of Belarus,
 220072 Minsk, Belarus}
 \email{koledad@rambler.ru}


\begin{document}

\maketitle

\begin{abstract}
In the paper we study the distribution of the discriminant $D(P)$ of polynomials $P$ from the class $\mathcal{P}_{n}(Q)$ of all integer polynomials of degree~$n$ and height at most~$Q$.
We evaluate the asymptotic number of polynomials $P\in \mathcal{P}_{n}(Q)$ having all the roots real and satisfying the inequality $|D(P)|\le X$ as $Q\to\infty$ and $X/Q^{2n-2}\to 0$.
\end{abstract}

\section*{Brief outline}

Section \ref{sec-intro} contains the background of the subject and the main results of the paper.
Section \ref{sec-aux} includes auxiliary proposions necessary to prove the main results; the reader can skip this section in the first reading.
In Section \ref{sec-count} we express the distribution function of the discriminant via the volumes of specific regions.
In Section \ref{sec-real} we study the asymptotic behaviour of the corresponding volume for small values of discriminant of polynomials having all the roots real.

Everywhere in the paper the degree $n\ge 4$ of polynomials is fixed; the parameter $Q$ bounding heights of polynomials from above grows to infinity.

\section{Introduction and main results}\label{sec-intro}

\subsection{Basic definitions}

The number
\[
 D(P)=a_{n}^{2n-2}\prod_{1\le i<j\le n}(\alpha_{i}-\alpha_{j})^{2}
\]
is called the discriminant of the polynomial
\[
P(x) = \sum_{k=0}^n a_k x^k = a_n \prod_{j=1}^n (x-\alpha_j).
\]

As a function of the polynomial coefficients, the discriminant can be written in the form of a $(2n-1)\times(2n-1)$ determinant:
\begin{equation*}
D(P) = (-1)^{\frac{n(n-1)}{2}}
\begin{vmatrix}
1 & a_{n-1} & \cdots & a_2 & a_1 & a_0 & & & \\
 & a_n & a_{n-1} & \cdots & a_2 & a_1 & a_0  & & \\
 & & \ddots & \ddots & \ddots & \ddots & \ddots & \ddots & \\
 & & & a_n & a_{n-1} & \cdots & a_2 & a_1 & a_0\\
n & \cdots & 3 a_3 & 2 a_2 & a_1 & & & & \\
 & n a_n & \cdots & 3 a_3 & 2a_2 & a_1 & & & \\
 & & \ddots & \ddots & \ddots & \ddots & \ddots & & \\
 & & & n a_n & \cdots & 3 a_3 & 2 a_2 & a_1 & \\
 & & & & n a_n & \cdots & 3 a_3 & 2 a_2 & a_1
\end{vmatrix},
\end{equation*}
where the empty entries are filled with zeros.

For given $n \in \mathbb{N}$ and $Q>1$, denote by $\mathcal{P}_{n}(Q)$ the set of all integer polynomials of degree $n$ and height $h(P)\le Q$.

For $v\ge 0$ define the following sets
\begin{gather}\label{eq-PQv}
\mathcal{P}_{n}(Q,v)=\left\{P\in \mathcal{P}_{n}(Q) : |D(P)|\le Q^{2n-2-2v}\right\},\\
\mathcal{P}^{(s)}_{n}(Q,v)=\left\{P\in \mathcal{P}_{n}(Q,v) : P \text{ has exactly $2s$ complex roots} \right\}. \notag
\end{gather}

Evidently,
\[
\mathcal{P}_{n}(Q,v) = \bigcup_{0\le s\le \frac{n}{2}} \mathcal{P}^{(s)}_{n}(Q,v).
\]

{\sc Remark.}
The most popular height functions in Number Theory are the naive height $H(P)=\max_{0\le k\le n} |a_k|$,
the Mahler measure $M(P)=a_n\prod_{j=1}^n \max(1,|\alpha_j|)$ and the length $L(P)=\sum_{k=0}^n |a_k|$ of the polynomial $P$.
Properties shared by these height functions are described in Subsection \ref{sbs-height}.
Due to these properties we are able to prove our main result (Theorem \ref{thm-small-discr} below) in the same way for all these height functions. So, we do not specify upon which height function $h$ the definition of $\mathcal{P}_{n}(Q)$
is based.

\subsection{Notation}
Here we explain the asymptotic notation we use.
The expression ``$f(x)\sim g(x)$ as $x\to x_0$'' is equivalent to $\lim_{x\to x_0} f(x)/g(x)=1$.
The statement ``$f(x)\ll_n g(x)$ as $x\to x_0$'' means the existence of a positive constant $c_n$ (depending on $n$ only) such that $|f(x)|\le c_n |g(x)|$ for all $x$ in a neighbourhood of $x_0$.
Note that $f(x)\ll g(x)$ is tantamount to $f(x)=O(g(x))$.
The notation $f(x)\asymp_n g(x)$ is a shortening for the two-sided asymptotic inequality $g(x)\ll_n f(x)\ll_n g(x)$.

For a finite set $S$ we denote its cardinality by $\#S$. For a set $S\subset \Bbb R^m$ the expression $\mes_d S$ denotes the $d$-dimensional Lebesgue measure of this set ($d\le m$).

To simplify the notation of quantities used in the paper, for some of them we do not indicate explicitly their dependence on $n$, since everywhere the degree $n$ is fixed.

\subsection{Background}

The discriminant of a polynomial characterizes at large the distances between the roots of the polynomial \cite{BM2010}, \cite{jhE2004}.
So the discriminant and its properties can be of big help in various problems, especially in the theory of Diophantine approximation.
See \cite{bV1961} as an example of such an application to the cubic case of the Mahler conjecture \cite{kM1932, vS1967}; the proof \cite{bV1961} is built upon Davenport's estimate \cite{hD1961} of the number of integer cubic polynomials having small discriminant.
This state of things raises big interest to the disctribution of the values of the polynomial discriminant.

There are a number of effective results aimed at problems related to algorithmic solution of Diophantine equations.
In this approach, the effective bounds on the extreme values of some quantities are usually of interest.
In 1970's Gy\H{o}ry established a series of such effective estimates \cite{kG1973, kG1974, kG1976, kG1978} concerning the distribution of the discriminants of integer polynomials; several of these results are improved in \cite{kG2006}.
A lot of further references and details on this topic can be found in the survey \cite{kG2006} by Gy\H{o}ry and the book \cite{EG2016} by Evertse and Gy\H{o}ry.

Another approach is concerned about the asymptotic statistics of polynomials with small discriminants. In this paper, we are mostly interested in problems from this second area. Now, we shortly tell about results directly related to the subject of the paper.

A possible way to state the problem about the distribution of polynomial discriminants is to ask to find lower and upper estimates, as close as possible, of the form
\[
Q^{f_*(v)} \ll_n \#\mathcal{P}_n(Q,v) \ll_n Q^{f^*(v)},
\]
where $f_*(v)$ and $f^*(v)$ are functions of $v$, with $0\le v\le n-1$.

There are plenty of papers dealing with estimating the quantity $\#\mathcal{P}_n(Q,v)$ and its close relatives and generalisations.

In 2008, Bernik, G\"otze and Kukso \cite{BGK08} proved that for $0<v<\frac12$
\[
\#\mathcal{P}_n(Q,v) \gg_n Q^{n+1-2v}.
\]

The upper bounds for degree $n=3$ one can find in \cite{Koleda2010}, \cite{KGK13}, \cite{BeBuGo2017};
on the asymptotics for $n=2$ see \cite{GKK13}.

Also there are papers concerning this problem in similar settings with $p$-adic norm (instead of usual absolute value) used to bound discriminant values in \eqref{eq-PQv} (see \cite{kG1978-p}, \cite{BGK-Litwa08}, \cite{BeBuDo2018}). See also the survey~\cite{BBGK13}.

For general $n$, the best up-to-date result regarding bounds on $\#\mathcal{P}_n(Q,v)$ is the one by Beresnevich, Bernik, G\"otze \cite{BBG15} who proved
\begin{equation}\label{eq-low-bound}
\#\mathcal{P}_n(Q,v) \gg_n Q^{n+1-\frac{n+2}{n} v} \quad \text{for } 0<v<n-1.
\end{equation}

In addition, using probabilistic methods, G\"otze and Zaporozhets \cite{fGdZ14} proved the existence of a continuous function $\phi_n$ such that
\begin{equation}\label{eq-dis-dis}
\sup_{-\infty\le a<b\le\infty} \left|\Bbb P\left\{a\le \frac{D(P)}{Q^{2n-2}}\le b\right\}-\int_a^b\phi_n(x)\,dx\right|
\ll_n \frac{1}{\log Q},
\end{equation}
where $\Bbb P\{A\}$ denotes the probability of an event $A$,
and the polynomial $P$ is picked at random uniformly from $\mathcal{P}_n(Q)$.

Here we prove that the difference in the left-hand side of \eqref{eq-dis-dis} can be estimated by $C_n/Q$ insead of $C_n/\log Q$.
Moreover, we get such results for every particular signature $s$.

Our main result is the exact asymptotics of $\#\mathcal{P}_n(Q,v)$ as $Q\to\infty$. In the present paper we extend the ideas from \cite{dK12-Trudy}, where the lower bound \eqref{eq-low-bound} was proved for $0<v<\frac{n}{n+2}$.
Another essential ingredient of our proof is the Selberg integral. We are going to prove that
\[
\#\mathcal{P}_n^{(0)}(Q,v) \asymp_n Q^{n+1-\frac{n+2}{n} v} \quad \text{for } 0<v<\frac{n}{n+2}.
\]

\subsection{Main results}

For $X\ge 0$, $Q>0$ and $0\le s\le \frac{n}{2}$, define the counting function
\[
N_s(Q,X) := \# \{ P \in \mathcal{P}_n(Q) : \text{$P$ has exaclty $s$ pairs of complex roots}, \ |D(P)| \le X \}.
\]
According to the following theorem, for every $s$ and sufficiently large $Q$ the overall shape of the function $N_s(Q,X)$ can be asymptotically described by a continuous function $f_s$.
\begin{theorem}\label{thm-lim-func}
Let $n\ge 3$ be an integer.
For every $s\ge 0$
there exists a positive continuous function $f_s:\RR\to\RR$ such that
\[
\left|N_s(Q,X) - Q^{n+1} f_s\!\left(\frac{X}{Q^{2n-2}}\right)\right|\ll_n Q^n,
\]
where the implicit constant depend on $n$ only.

The function $f_s(\delta)$ is monotonically increasing as $\delta$ grows, and $\lim_{\delta\to 0} f_s(\delta)=0$.
\end{theorem}

Obviously, if a polynomial $P$ is picked at random uniformly from $\mathcal{P}_n(Q)$, then for $\delta\ge 0$
\[
\Bbb P\left\{\frac{|D(P)|}{Q^{2n-2}}\le \delta \right\}
= \frac{1}{\#\mathcal{P}_n(Q)} \sum_{0\le s\le \frac{n}{2}} N_s(Q,\delta Q^{2n-2}).
\]
Moreover, since $\#\mathcal{P}_n(Q)=2Q(2Q+1)^n$, for the function $\phi_n$ in \eqref{eq-dis-dis} we have
\[
\int_{-\delta}^{\delta} \phi_n(x) dx= 2^{-n-1}\sum_{0\le s\le n/2} f_s(\delta),
\]
Thus, Theorem \ref{thm-lim-func} shows that the difference in \eqref{eq-dis-dis} can be estimated as $C_n Q^{-1}$, where the constant $C_n$ depends on $n$ only.

In the case $s=0$, that is, when all the roots are real, we prove the 

\begin{theorem}\label{thm-small-discr}
Let $n\ge 2$ be an integer. In the totally real case (that is, for $s=0$) we have
\[
f_0(\delta) \sim \lambda_0\,\delta^{\frac{n+2}{2n}}, \quad \text{as} \quad \delta\to 0.
\]
The constant $\lambda_0$ can be computed explicitly and depends only on $n$ and the height function~$h$.
\end{theorem}

{\sc Remark.}
The corresponding results for $n=2$ and $n=3$ are obtained in \cite{GKK13} and \cite{KGK13}.

Theorem \ref{thm-small-discr} shows that $\lim_{x\to+0} \phi_n(x)=+\infty$.

Since $\#\mathcal{P}_n^{(s)}(Q,v)=N_s(Q,Q^{2n-2-2v})$, we get the following corollary.
\begin{cor}
Let $\epsilon\in(0,1)$ be a fixed small number. For all $v\in\left[\epsilon,(1-\epsilon)\frac{n}{n+2}\right]$ we have
\[
\#\mathcal{P}^{(0)}_n(Q,v) \sim \lambda_0\, Q^{n+1-\frac{n+2}{n} v},  \quad \text{as} \quad Q\to \infty,
\]
where $\lambda_0$ is the same as in Theorem \ref{thm-small-discr}.
\end{cor}

\section{Auxiliary statements}\label{sec-aux}

\begin{lemma}[Davenport \cite{hD51}]\label{lm-int-p-num}
Let $\mathcal{D}\subset \mathbb{R}^d$ be a bounded region formed by points $(x_1,\dots,x_d)$ satisfying a finite collection of algebraic inequalities
\[
F_i(x_1,\dots,x_d)\ge 0, \qquad 1\le i\le k,
\]
where $F_i$ is a polynomial of degree $\deg F_i \le m$ with real coefficients.
Let
\[
\Lambda(\mathcal{D}) = \mathcal{D}\cap \mathbb{Z}^d.
\]
Then
\[
\left|\#\Lambda(\mathcal{D}) - \mes_d \mathcal{D}\right| \le C \max(\bar{V}, 1),
\]
where the constant $C$ depends only on $d$, $k$, $m$; the quantity $\bar{V}$ is the maximum of all $r$--dimensional measures of projections of $\mathcal{D}$ onto all the~coordinate subspaces obtained by making $d-r$ coordinates of points in $\mathcal{D}$ equal to zero, $r$ taking all values from $1$ to $d-1$, that is,
\[
\bar{V}(\mathcal{D}) := \max\limits_{1\le r < d}\left\{ \bar{V}_r(\mathcal{D}) \right\}, \quad
\bar{V}_r(\mathcal{D}) := \max\limits_{\substack{\mathcal{J}\subset\{1,\dots,d\} \\ \#\mathcal{J} = r}}\left\{ \mes_r \operatorname{Proj}_{\mathcal{J}} \mathcal{D} \right\},
\]
where $\operatorname{Proj}_{\mathcal{J}} \mathcal{D}$ is the orthogonal projection of $\mathcal{D}$ onto the coordinate subspace formed by coordinates with indices in $\mathcal{J}$.
\end{lemma}

\begin{lemma}\label{lm-int-below}
Let $f: \RR^{n+1}\to\RR$ be a integrable function such that for any compact set $S\subset\RR^n$
\[
\lim_{\delta\to 0} \sup_{\mathbf{x}\in S}|f(\delta,\mathbf{x})-f(0,\mathbf{x})| = 0.
\]
Let $|f(\delta,\vx)|\le g(\vx)$ for $|\delta|<\delta_0$ and all $\vx\in\RR^n$, and the integral
$\int_{\RR^n} g(\vx) d\vx$ converge.

Then the following limit has a finite value:
\[
\lim_{\delta\to 0}\int_{\RR^n} f(\delta,\vx) d\vx = \int_{\RR^n} f(0,\vx) d\vx.
\]
\end{lemma}

{\sc Remark.}
This lemma is merely an adjusted combination of classical results, which can be found in many calculus textbooks. We prove this statement here for the reader's convenience.

\begin{proof}
Evidently, the integral $\int_{\RR^n} f(\delta,\vx)\, d\vx$ converges for all $|\delta|\le \delta_0$ because $|f(\delta,\vx)|$ is bounded by $g(\vx)$.

Moreover, for any $R>0$ we have 
\[
\left|\int_{\RR^n} f(\delta,\vx) d\vx - \int_{\RR^n} f(0,\vx) d\vx\right| \le \int_{|x_i|\le R} \left| f(\delta,\vx) - f(0,\vx) \right| d\vx +
2 \int_{|x_i| > R} g(\vx) d\vx.
\]
Both integrals in the right-hand side can be made arbitrarily small by an appropriate choice $R$ and $\delta$.
Therefore, the lemma is proved.
\end{proof}

\begin{lemma}\label{lm-int-form}
Let $f:\RR^m\to\RR$ be a positive integrable function, and let for any $r>0$ and all $\vx\in\RR^m$
\[
f(r\vx) = r^c f(\vx),
\]
where $c$ is a positive constant.

Let $d_1$, $d_2$ be real numbers such that
\[
cd_1 < -m < cd_2.
\]
Let $D=\left\{(x_1,\dots,x_m)\in\RR^m : x_m > |x_i|\right\}$.
Let $\kappa>0$ be a constant.

Then
\begin{multline}\label{eq-lm2}
\int_D \min\left(f(\vx)^{d_1}, \kappa f(\vx)^{d_2}\right) d\vx =\\
=\kappa^{-\frac{m+cd_1}{c(d_2-d_1)}} \left(\frac{1}{m+cd_2}-\frac{1}{m+cd_1}\right)
\int_{[-1,1]^{m-1}} f(t_1,\dots,t_{m-1},1)^{-m/c}\,dt_1\dots dt_{m-1},
\end{multline}
and the left-hand-side integral converges if and only if the right-hand-side one does.
\end{lemma}

\begin{proof}
We start with proving \eqref{eq-lm2} for $\kappa=1$, that is, we obtain the formula
\begin{multline}\label{eq-lm2-K1}
\int_D \min\left(f(\vx)^{d_1}, f(\vx)^{d_2}\right) d\vx =\\
\left(\frac{1}{m+cd_2}-\frac{1}{m+cd_1}\right)
\int_{[-1,1]^{m-1}} f(t_1,\dots,t_{m-1},1)^{-m/c}\,dt_1\dots dt_{m-1}.
\end{multline}

Comparing the functions under the minimum in \eqref{eq-lm2-K1}, we split the left-hand-side integral into the sum of two integrals
\[
\int_D = \int_{D_1} f(\vx)^{d_2}d\vx + \int_{D_2} f(\vx)^{d_1}d\vx.
\]
where
\begin{align*}
D_1 &= \{\vx\in D : f(\vx)\le 1\}, \\
D_2 &= \{\vx\in D : f(\vx)> 1\}.
\end{align*}
The integral over $D_1$ will be denoted as $I_1$, the one over $D_2$ --- as $I_2$.

Let's change the variables: $x_m = r$, $x_i = r t_i$ for $1\le i < m$.
The Jacobian is equal to $r^{m-1}$. Then
\[
I_1 = \int_{\widetilde{D}_1} r^{m-1+cd_2} f(t_1,\dots,t_{m-1},1)^{d_2} \,dr\,dt_1\dots dt_{m-1},
\]
where
\[
\widetilde{D}_1 = \left\{(t_1,\dots,t_{m-1},r) \in \RR^m : |t_i|\le 1, \ 0\le r \le f(t_1,\dots,t_{m-1},1)^{-1/c} \right\}.
\]
Applying Fubini's theorem and integrating with respect to $r$, we get
\[
I_1 = \frac{1}{m+c d_2} \int_{|t_i|\le 1} f(t_1,\dots,t_{m-1},1)^{-m/c} \,dt_1\dots dt_{m-1}.
\]
For $I_2$ we get
\[
I_2 = \int_{\widetilde{D}_2} r^{m-1+cd_1} f(t_1,\dots,t_{m-1},1)^{d_1} \,dr\,dt_1\dots dt_{m-1},
\]
where
\[
\widetilde{D}_2 = \left\{(t_1,\dots,t_{m-1},r) \in \RR^m : |t_i|\le 1, \  r > f(t_1,\dots,t_{m-1},1)^{-1/c} \right\}.
\]
Note that from the lemma statement $m-1+c d_1 < -1$. So, we have
\[
I_2 = - \frac{1}{m+cd_1} \int_{|t_i|\le 1} f(t_1,\dots,t_{m-1},1)^{-m/c} \,dt_1\dots dt_{m-1}.
\]
So, \eqref{eq-lm2-K1} is proved.
Now, note that
\[
\min\left(f(\vx)^{d_1}, \kappa f(\vx)^{d_2}\right) = \kappa^{\frac{-d_1}{d_2-d_1}} \min\left(\kappa^{\frac{d_1}{d_2-d_1}} f(\vx)^{d_1}, \kappa^{\frac{d_2}{d_2-d_1}} f(\vx)^{d_2}\right).
\]
Hence, putting the function $\tilde{f}(\vx) := \kappa^{\frac{1}{d_2-d_1}} f(\vx)$ in \eqref{eq-lm2-K1} and multiplying the result by $\kappa^{\frac{-d_1}{d_2-d_1}}$ finish the proof of the lemma.
\end{proof}

\begin{lemma}\label{lm-jacob}
Let $n\ge 2$ be an integer.
Consider the change between the real variables
\[
(a_0,a_1,\dots,a_n) \quad \text{and} \quad
(b,z_1,\dots,z_n)
\]
given by the relation
\begin{equation}\label{eq-p-repr}
\sum_{k=1}^n a_k x^k = b\prod_{j=1}^n (x-z_j).
\end{equation}

Then the Jacobian of this change equals to:
\begin{equation}
\left| \frac{\partial (a_0,a_1,\dots,a_n)}{\partial (b;z_1,\dots,z_n)} \right| = b^n \!\!\prod_{1\le i<j\le n} |z_i - z_j|.
\end{equation}
\end{lemma}
\begin{proof}
For a finite set $\mcA\subset\Bbb N$ let us denote
\[
\sigma_k[\mcA]:=\sum_{1\le j_1<\dots<j_k} z_{j_1}\dots z_{j_k}, \qquad \text{where} \quad z_j:=0 \quad \text{if} \quad j\not\in\mcA.
\]
Take by definition $\sigma_0[\mcA]:=1$ for any $\mcA\subset\Bbb N$ (including $\mcA=\varnothing$).

For any $j\in\Bbb N$ and $k\ge 1$ we have
\begin{equation}\label{eq-sigma-id}
\sigma_k[\mcA]=z_j \sigma_{k-1}[\mcA\setminus\{j\}]+ \sigma_k[\mcA\setminus\{j\}].
\end{equation}

The equality
\eqref{eq-p-repr}
is equivalent to
\[
a_k = (-1)^{n-k} b\,\sigma_{n-k}[\mcA], \qquad \text{with} \quad \mcA=\{1,2,\dots,n\}.
\]

With the help of \eqref{eq-sigma-id} one can express the determinant $J_0$ of the Jacobi matrix in the following way
\[
J_0 :=
\begin{vmatrix}
\pder{a_n}{b} & \pder{a_n}{z_1} & \dots & \pder{a_n}{z_n}\\
\vdots & \vdots & \cdots & \vdots\\
\pder{a_0}{b} & \pder{a_0}{z_1} & \dots & \pder{a_0}{z_n}\\
\end{vmatrix}
= (-1)^{\frac{n(n+1)}{2}}\, b^n J[\mcA],
\]
where
\[
J[\mcA]=\begin{vmatrix}
\sigma_0[\mcA\setminus\{1\}] & \sigma_0[\mcA\setminus\{2\}] & \dots & \sigma_0[\mcA\setminus\{n\}]\\
\sigma_1[\mcA\setminus\{1\}] & \sigma_1[\mcA\setminus\{2\}] & \dots & \sigma_1[\mcA\setminus\{n\}]\\
\vdots & \vdots & \ddots & \vdots\\
\sigma_{n-1}[\mcA\setminus\{1\}] & \sigma_{n-1}[\mcA\setminus\{2\}] & \dots & \sigma_{n-1}[\mcA\setminus\{n\}]\\
\end{vmatrix}.
\]
Subtracting the first column from the others and using \eqref{eq-sigma-id}, one can easily check that
\[
J[\mcA]=J[\mcA_{-1}] \prod_{j=2}^n (z_1-z_j),
\qquad \text{where} \quad \mcA_{-1}:=\mcA\setminus\{1\}.
\]
Repeating this reduction procedure, finally we obtain the last nontrivial determinant for $\mcA_{-(n-2)}:=\{n-1,n\}$
\[
J[\mcA_{-(n-2)}]=\begin{vmatrix}
\sigma_0[\mcA_{-(n-2)}\setminus\{n-1\}] & \sigma_0[\mcA_{-(n-2)}\setminus\{n\}]\\
\sigma_1[\mcA_{-(n-2)}\setminus\{n-1\}] & \sigma_1[\mcA_{-(n-2)}\setminus\{n\}]
\end{vmatrix}=
\begin{vmatrix}
1 & 1\\
z_n & z_{n-1}
\end{vmatrix}.
\]
Hence we have
\[
J_0=(-1)^{\frac{n(n+1)}{2}}\, b^n \prod_{1\le j<k\le n} (z_j-z_k).
\]
Taking the absolute value in the latter equation finishes the proof.
\end{proof}

\begin{lemma}\label{lm-cont-meas}
Let $X$ be a measurable set, and $g:X\to\Bbb R$ a measurable function.
Then the function $G:\Bbb R\to\Bbb R$ defined by
$G(a):=\mes\{x\in X: g(x)\le a\}$
is continuous if and only if
$\mes\{x\in X: g(x)=a\}=0$
for all real $a$.
\end{lemma}
\begin{proof}
Obviously, $\mes\{x\in X: g(x)=a\}=G(a)-G(a-0)$.
In addition, $G(a+0)=\lim_{\epsilon\to+0} G(a+\epsilon)=G(a)$.
The lemma follows.
\end{proof}

\section{Counting integer polynomials}\label{sec-count}

In this section we reduce counting integer polynomials to evaluating of volumes of some regions.

\subsection{Height functions}\label{sbs-height}

There are a number of height functions defined on real polynomials, e.g. naive height, Mahler measure, length etc. Our result for each of them can be obtained in the same way.
So, we summarize their essential properties in a general notion of a height function and prove our theorem in a general form.

A continuous function $h:\RR^{n+1}\to[0,+\infty)$ satisfying for all $\vv=(v_n,\dots,v_1,v_0)\in\RR^{n+1}$ the properties
\begin{enumerate}
\item $h(t\vv)=|t|h(\vv)$ for all real $t$; \label{pr-hom}

\item $h(\vv)=0$ if and only if $\vv=\mathbf{0}$;

\item $h(v_0,v_1,\dots,v_n)=h(v_n,\dots,v_1,v_0)$; \label{pr-inv}

\item $h(v_0,-v_1,\dots,(-1)^k v_k,\dots,(-1)^n v_n) = h(v_0,v_1,\dots,v_k,\dots,v_n)$, \label{pr-minus}
\end{enumerate}
will be called a \emph{height function}.

Note that in terms of polynomials the property \ref{pr-inv} is equivalent to $h(x^n P(x^{-1}))=h(P(x))$.
The property \ref{pr-minus} is equivalent to $h(P(-x))=h(P(x))$.

We use the same notation for the height of a vector and of a polynomial.

\begin{lemma}\label{lm-h-bound}
For any $\vv\in\RR^{n+1}\setminus (-1,1)^{n+1}$ we have
$h(\vv)\ge h_0$,
where $h_0$ is a positive constant depending only on the height function $h$ and the parameter $n$.
\end{lemma}
\begin{proof}
The set $S_0:=\{\vv\in\RR^{n+1}: \|\vv\|_\infty = 1\}$ is compact. From the extreme value theorem we have that there exists a value $h_0$ such that $h(\vv)\ge h_0$ for all $\vv\in S_0$.

Now, noticing that $h(\vv)=\|\vv\|_\infty h(\vv_0)$, where $\vv_0:=\|\vv\|_\infty^{-1} \vv\in S_0$, and $\|\vv\|_\infty \ge 1$ for all $\vv\in\RR^{n+1}\setminus(-1,1)^{n+1}$, we have the lemma.
\end{proof}

\subsection{Counting integer points via volume}
We represent real polynomials of degree~$n$ by their vectors of coefficients in $\RR^{n+1}$.
So, to calculate $N_s(Q,X)$, we need to count integral points in the region
\[
\mathcal{D}_s:=\left\{\va\in\RR^{n+1}: \sum_{i=0}^n a_j x^j \text{ has $s$ pairs of complex roots},\ h(\va)\le Q, \ |D(\va)|\le X\right\},
\]
where $D(\va)$ is the discriminant as a function of the coefficients $\va$ of the polynomial $\sum_{i=0}^n a_j x^j$. Note that $D(a_0,\dots,a_n)$ is a homogeneous polynomial of degree $2n-2$ in variables $a_0,\dots,a_n$.

From Lemma \ref{lm-int-p-num} we have
\[
\left|\# (\mathcal{D}_s \cap \ZZ^{n+1}) - Q^{n+1} \mes_{n+1}\widetilde{\mathcal{D}}_s\left(\frac{X}{Q^{2n-2}}\right)\right| \le c\,Q^n,
\]
where $c$ is a constant depending on $n$ only, and
\[
\widetilde{\mathcal{D}}_s(\delta):=\left\{\va\in\RR^{n+1}: \sum_{i=0}^n a_j x^j \text{ has $s$ pairs of complex roots},\ h(\va)\le 1, \ |D(\va)|\le \delta\right\}.
\]
Obviously, the function $f_s(\delta)$ in Theorem \ref{thm-lim-func} must equal to $\mes_{n+1}\widetilde{\mathcal{D}}_s(\delta)$:
\[
f_s(\delta):=\mes_{n+1}\widetilde{\mathcal{D}}_s(\delta).
\]

Since $\mes_{n+1}\varnothing = 0$, it is natural to set $f_s(\delta):=0$ for negative $\delta$.
Now note that for all real $X$
\[
\mes_{n+1}\left\{\va\in\RR^{n+1}: h(\va)\le 1, \ D(\va)=X\right\}=0
\]
because $D(\va)$ is a nonconstant polynomial in the variables $\va$.
Applying Lemma \ref{lm-cont-meas} we see that $f_s(\delta)$ is a continuous function, and
$f(0)=\lim_{\delta\to 0} f(\delta)=0$.
The monotonicity of $f_s$ is obvious.

\section{Proof of Theorem \ref{thm-small-discr}}\label{sec-real}

In this section we find the asymptotics of $f_0(\delta)=\mes_{n+1}\widetilde{\mathcal{D}}_0(\delta)$ as $\delta\to +0$.

By definition
\[
\mes_{n+1}\widetilde{\mathcal{D}}_0(\delta) = \int_{\widetilde{\mathcal{D}}_0(\delta)} da_0\, da_1\dots da_n.
\]
Changing the variables $(a_0,\dots,a_n)$ in this integral to $(b;\alpha_1,\alpha_2,\dots,\alpha_n)$, where $b=a_n$, and $\alpha_i$'s are roots of $\sum_{j=0}^n a_j x^j$ (see Lemma \ref{lm-jacob}), we obtain
\[
\mes_{n+1} \widetilde{\mathcal{D}}_0(\delta)=\frac{1}{n!} \int_{B_\delta} b^n \sqrt{\Delta(\valpha)}\, db\, d\valpha,
\]
where for $\valpha=(\alpha_1,\dots, \alpha_n)\in\RR^n$
\[
\Delta(\valpha) = \prod_{1\le i<j\le n} (\alpha_i - \alpha_j)^2,
\]
and the region $B_\delta \subset\RR^{n+1}$ is defined by
\begin{equation}\label{eq-B-def}
B_\delta := \left\{
(b;\valpha)\in\RR^{n+1} :
|b| h(p_{\valpha}) \le 1,\
b^{2n-2} \Delta(\valpha) \le \delta
\right\},
\qquad p_{\valpha}(x)=\prod_{i=1}^n (x-\alpha_i).
\end{equation}
Here we took into account that the Jacobian of this variable change equals to $b^n \sqrt{\Delta(\valpha)}$.
The factor $(n!)^{-1}$ arises because of the symmetry of the roots $\alpha_1,\dots,\alpha_n$.

Denote $K(\valpha):=(h(p_{\valpha}))^{-1}$.
The two restrictions on $b$ in \eqref{eq-B-def} can be joined in the single inequality $|b|\le \Psi(\valpha)$, where
\[
\Psi(\valpha) = \min\left\{ \left(\frac{\delta}{\Delta(\valpha)}\right)^{\frac{1}{2n-2}}, \, K(\valpha)\right\}.
\]

We need to evaluate the integral
\[
J = \int_{B_\delta} b^n \sqrt{\Delta(\valpha)}\, db\, d\valpha = \frac{2}{n+1} \int_{\RR^n} \sqrt{\Delta(\valpha)}\, \Psi(\valpha)^{n+1} d\valpha.
\]
Note that both these integrals converge since
$\mes_{n+1}\mathcal{D}_0(\delta)\le \mes_{n+1} \left\{\va\in\RR^{n+1}: h(\va)\le 1\right\}$.

To ensure the correctness of the transformations below, we restrict the range of one of the roots $\alpha_i$.
\begin{lemma}\label{lm-bounded-region}
\begin{equation}\label{eq-restrain}
\int_{\RR^n} \sqrt{\Delta(\valpha)}\, \Psi(\valpha)^{n+1} d\valpha = 2\int_{[-1,1]\times\RR^{n-1}} \sqrt{\Delta(\valpha)}\, \Psi(\valpha)^{n+1} d\valpha.
\end{equation}
\end{lemma}
\begin{proof}
Let us make the change of variables
$\alpha_i = \beta_i^{-1}$ in the latter integral.
For this change the Jacobian equals
\begin{equation}\label{eq-jac-lemma}
\det \left(\frac{\partial\alpha_i}{\partial\beta_j}\right)_{i,j=1}^n = \left(\prod_{k=1}^n \beta_k\right)^{-2}.
\end{equation}

We have
\begin{align*}
\Delta(\valpha) &= \prod_{1\le i<j\le n} \left(\frac{1}{\beta_i} - \frac{1}{\beta_j}\right)^2 = \left(\prod_{k=1}^n \beta_k\right)^{-(2n-2)} \Delta(\vbeta),\\
\sigma_k(\valpha) &= \left(\prod_{k=1}^n \beta_k\right)^{-1} \sigma_{n-k}(\vbeta).
\end{align*}
So, from the properties of the height function, we have $K(\valpha)=K(\vbeta) \left|\prod_{k=1}^n \beta_k\right|$, and thus,
\[
\Psi(\valpha) = \Psi(\vbeta) \left|\prod_{k=1}^n \beta_k\right|.
\]
Now, we have
\[
\int_{[-1,1]\times\RR^{n-1}} \sqrt{\Delta(\valpha)}\, \Psi(\valpha)^{n+1} d\alpha_1\dots d\alpha_n =
\int_{(\RR\setminus[-1,1])\times\RR^{n-1}} \sqrt{\Delta(\vbeta)}\, \Psi(\vbeta)^{n+1} d\beta_1\dots d\beta_n.
\]
The lemma is proved.
\end{proof}

Using this lemma, we get
\[
J = \frac{4}{n+1} I_1, \qquad I_1 := \int_{[-1,1]\times\RR^{n-1}} \sqrt{\Delta(\valpha)}\, \Psi(\valpha)^{n+1} d\valpha.
\]
Now, make one more change of variables with a parameter $\rho$:
\begin{equation}\label{eq-center-ch}
\begin{aligned}
\alpha_1 &= \tau,\\
\alpha_i &= \tau + \rho\,\theta_{i-1}, \qquad 2\le i\le n.
\end{aligned}
\end{equation}
The Jacobian of this change equals $\rho^{n-1}$. The parameter $\rho$ will be specified later.
\[
\Delta(\valpha) = \rho^{n(n-1)} \prod_{k=1}^{n-1} \theta_k^2 \prod_{1\le i<j\le n-1} (\theta_i-\theta_j)^2.
\]
For $\vtheta=(\theta_1,\dots,\theta_{n-1})$ denote
\[
\tDelta(\vtheta) := \prod_{k=1}^{n-1} \theta_k^2 \prod_{1\le i<j\le n-1} (\theta_i-\theta_j)^2,
\]
so that $\Delta(\valpha) = \rho^{n(n-1)} \tDelta(\vtheta)$.

Putting the new variables into $\Psi(\valpha)$, we see that $\rho^{n(n-1)}=\delta$ is the best choice for $\rho$. So, we have
\[
\Psi(\valpha) = \tPsi(\rho;\tau,\vtheta),
\]
where
\[
\tPsi(\rho;\tau,\vtheta) := \min\left\{ \tDelta(\vtheta)^{\frac{-1}{2n-2}}, \, K(\vtau+\rho\,\hat{\vtheta})\right\},
\]
where $\vtau:=(\tau,\dots,\tau)\in\RR^n$, $\hat{\vtheta}:=(0,\theta_1,\dots,\theta_{n-1})\in\RR^n$.

Now, we have
\begin{multline*}
I_1 = \int_{-1}^1 d\tau\int_{\RR^{n-1}} \rho^{\frac{n(n-1)}{2}} \sqrt{\tDelta(\vtheta)}\, \tPsi(\rho;\tau,\vtheta)^{n+1} \rho^{n-1} d\theta_1\dots d\theta_{n-1} =\\
= \rho^{\frac{(n-1)(n+2)}{2}} \int_{-1}^1 d\tau\int_{\RR^{n-1}} \sqrt{\tDelta(\vtheta)}\, \tPsi(\rho;\tau,\vtheta)^{n+1} d\theta_1\dots d\theta_{n-1}.
\end{multline*}

Now the question is: does the following integral have a finite limit (as $\rho$ vanishes) or not?
And if it does, then what value does it have.
\[
\tI_1(\rho) = \int_{-1}^1 d\tau\int_{\RR^{n-1}} \sqrt{\tDelta(\vtheta)}\, \tPsi(\rho;\tau,\vtheta)^{n+1} d\theta_1\dots d\theta_{n-1}.
\]

Let us denote
\begin{align}
I_2(\rho,\tau)&:=\int_{\RR^{n-1}} \sqrt{\tDelta(\vtheta)}\, \tPsi(\rho;\tau,\vtheta)^{n+1} d\theta_1\dots d\theta_{n-1},\\
I_3(K)&:=\int_{\RR^{n-1}} \sqrt{\tDelta(\vtheta)}\, \min\left\{K, \tDelta(\vtheta)^{\frac{-1}{2n-2}}\right\}^{n+1} d\theta_1\dots d\theta_{n-1}. \label{eq-int}
\end{align}

For any $\rho$, $\tau$, $\vtheta$ we have $K(\vtau+\rho\,\hat{\vtheta})\le K_0$,
where $K_0:=\sup_{\valpha\in\RR^n} K(\valpha)$. The value $K_0$ is finite because of Lemma \ref{lm-h-bound}.
Therefore,
\begin{equation}\label{eq-unif-bnd}
I_2(\rho,\tau)\le I_3(K_0).
\end{equation}

\begin{lemma}\label{lm-selberg}
The integral $I_3(K)$ converges and can be written in the form
\begin{equation*}
I_3(K) = \frac{n(n+1)}{(n+2)}\,K^{2/n} \!\!\int\limits_{[-1,1]^{n-2}} \left(\prod_{i=1}^{n-2} |\xi_i| \,(1-\xi_i)
\!\!\prod_{1 \le i < j \le n-2} |\xi_i - \xi_j | \right)^{-2/n}
\!\!d\xi_1 \cdots d\xi_{n-2}.
\end{equation*}
\end{lemma}

\begin{proof}
The function $\tDelta(\vtheta)$ is a symmetric homogeneous polynomial of total degree $c=n(n-1)$. It is not difficult to see that the integral \eqref{eq-int} meets the conditions of Lemma \ref{lm-int-form} with $m=n-1$, $d_1=-\frac{1}{n-1}$ and $d_2=\frac{1}{2}$, and $\kappa=K^{n+1}$. Applying Lemma \ref{lm-int-form} (with $cd_2-cd_1=\frac{n(n+1)}{2}$, $m+cd_1=-1$, and $m+cd_2=\frac{(n-1)(n+2)}{2}$; note, we must multiply the result by $m$),
we see that
\begin{equation}\label{eq-int1}
I_2 = \frac{n(n+1)}{(n+2)}\,K^{2/n} \int_{[-1,1]^{n-2}} \tDelta(\xi_1,\dots,\xi_{n-2},1)^{-1/n} d\xi_1\dots d\xi_{n-2},
\end{equation}
and, assuming that the integral \eqref{eq-int} converges,
the latter integral must converge too.
In its expanded form, this integral can written as
\begin{equation}\label{eq-int2}
I_4 = \int\limits_{[-1,1]^{n-2}} \prod_{i=1}^{n-2} |\xi_i|^{-2/n}(1-\xi_i)^{-2/n}
\prod_{1 \le i < j \le n-2} |\xi_i - \xi_j |^{-2/n}\,d\xi_1 \cdots d\xi_{n-2}.
\end{equation}

Now, we are going to ``selbergate'' (named after ``[Atle] Selberg + [integr]ate''), that is, to reduce the solution of the problem to the famous Selberg integral \cite{Sel1944}:
\begin{equation*}
S_m (\alpha, \beta, \gamma) =
\int_0^1 \cdots \int_0^1 \prod_{i=1}^m t_i^{\alpha-1}(1-t_i)^{\beta-1}
\prod_{1 \le i < j \le m} |t_i - t_j |^{2 \gamma}\,dt_1 \cdots dt_m.
\end{equation*}
For any integer $m\ge 1$, the Selberg integral converges for
\begin{equation}\label{eq-Sel-conv-c}
\Re(\alpha) > 0, \ \Re(\beta) > 0, \ \Re(\gamma) > -\min\left\{\frac{1}{m}, \frac{\Re(\alpha)}{m - 1}, \frac{\Re(\beta)}{m - 1}\right\}.
\end{equation}

If one takes a look at the integral \eqref{eq-int2}, they will see that
\[
I_4 \le 2^{n-2} S_{n-2}\left(\frac{n-2}{n},\, \frac{n-2}{n},\, -\frac{1}{n}\right).
\]
Obviously, the convergence conditions \eqref{eq-Sel-conv-c} are satisfied:
\[
-\frac{1}{n} > -\min\left\{\frac{1}{n-2},\ \frac{n-2}{n(n - 3)}\right\}.
\]
Hence $S_{n-2}\left(\frac{n-2}{n},\, \frac{n-2}{n},\, -\frac{1}{n}\right)$ converges, and so does $I_4$. Thus, $I_3(K)$ converges too.
\end{proof}

On the other hand, Lemma \ref{lm-int-below} tells us that
\[
\lim_{\rho\to 0} I_2(\rho,\tau)=I_2(0,\tau)=I_3(\tK(\tau)),
\]
where $\tK(\tau):=K(\vtau)$.
This convergence is uniform for $\tau\in[-1,1]$ because of \eqref{eq-unif-bnd}. Thus,
\begin{multline*}
\tI_1(0)=\lim_{\rho\to 0} \tI_1(\rho) = \int_{-1}^1 I_3(\tK(\tau))\,d\tau=\\
=\frac{n(n+1)}{(n+2)} \int_{-1}^1 \tK(\tau)^{2/n} d\tau \int_{[-1,1]^{n-2}} \tDelta(\xi_1,\dots,\xi_{n-2},1)^{-1/n} d\xi_1\dots d\xi_{n-2}.
\end{multline*}
And finally, we see that
\[
\lambda_0=\frac{4}{(n+1)!}\, \tI_1(0).
\]
Theorem \ref{thm-small-discr} is proved.

\bibliographystyle{abbrv}
\bibliography{discr4v1}

\end{document}